\documentclass[a4paper,11pt]{article}
\usepackage[utf8]{inputenc}
\usepackage[T1]{fontenc}
\usepackage[english]{babel}
\usepackage[margin=2.8cm]{geometry}
\usepackage{mathpazo}
\usepackage[scaled=.95]{helvet}
\usepackage{courier}
\usepackage{amsmath,amsfonts,amssymb,amsthm}
\usepackage{graphicx}
\usepackage{natbib}
\usepackage{color}

\usepackage{tikz}
\usetikzlibrary{intersections}
\usepackage[draft]{fixme}
\fxusetheme{color}
\FXRegisterAuthor{pa}{apa}{PA}
\FXRegisterAuthor{pe}{ape}{P}
\FXRegisterAuthor{h}{ah}{H}

\newcommand{\Rlogo}{\protect\includegraphics[height=1.8ex,keepaspectratio]{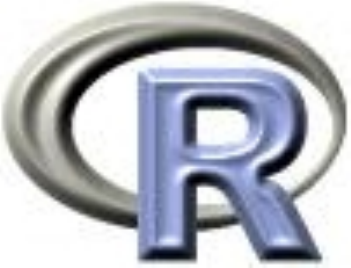}}
\newcommand{\egaldef}{\mathrel{:=}}
\newcommand*{\CF}{\mathcal{F}}
\newcommand*{\Fn}{\CF_n}
\newcommand*{\E}{\mathbb{E}}
\newcommand*{\EE}[1]{\E\left[#1\right]}
\newcommand*{\PP}{\mbox{$\mathbb{P}$}}

\newcommand*{\ECond}[2]{\E\left[#1\middle| #2\right]}
\newcommand*{\EFn}[1]{\ECond{#1}{\CF_n}}

\newcommand*{\R}{\mathbb{R}}
\newcommand*{\xR}{\mathbb{R}}
\newcommand*{\xN}{\mathbb{N}}
\newcommand*{\scal}[1]{\left\langle #1 \right\rangle} 
\newcommand*{\nrm}[1]{\left\| #1 \right\|}            
\newcommand*{\nrmOp}[1]{\nrm{#1}_{\text{op}}}            
\newcommand*{\nrmHS}[1]{\nrm{#1}_{\text{\tiny H.S.}}}            
\newcommand*{\abs}[1]{\left| #1 \right|}              
\newcommand*{\Var}{\mathbf{Var}}  		
\newcommand*{\Cov}{\mathbf{Cov}}                    

\newcommand*{\argmin}{\mathbf{argmin}}

\newcommand*{\ind}[1]{\mathbf{1}_{#1}}
\newcommand{\CL}{\mathcal{L}}

\newcommand{\ep}{\varepsilon} 

\newcommand{\idt}{\mathbf{I}_H}

\newcommand{\Zbar}{\overline{Z}}
\newcommand{\Tbar}{\overline{T}}

\newcommand*{\wass}[1]{\mathcal{W}_2\left(#1\right)}
\newcommand*{\law}{\mathcal{L}}

\newcommand*{\cstHolder}{C_3}
\newcommand*{\cstWass}{C_4}
\newcommand*{\cstMoment}{C_6}
\newcommand*{\assm}[1]{\textbf{A#1}}

\newcommand{\lmin}{\lambda_{min}}
\newcommand{\cG}{c_\gamma}
\newcommand{\covLimite}{\Sigma}
\newcommand{\avgCov}{\overline{\Sigma}}

\newcommand{\limite}[2]{\xrightarrow[#1]{#2}}

\DeclareMathOperator{\cv}{\limite{n \to \infty}{}}
\DeclareMathOperator{\cvp}{\limite{n \to \infty}{P}}
\DeclareMathOperator{\cvl}{\limite{n \to \infty}{\CL}}

\DeclareMathOperator{\trace}{Tr}

\newtheorem{defn}{Definition}
\newtheorem*{notation}{Notation}

\newtheorem{thm}{Theorem}[section]

\newtheorem{lem}[thm]{Lemma}
\newtheorem{prop}[thm]{Proposition}

\newtheorem{rem}{Remark}

\begin{document}

\title{Recursive estimation of the conditional geometric  median in Hilbert spaces}

\author{Herv\'e \textsc{Cardot}, Peggy \textsc{C\'enac}, Pierre-Andr\'e \textsc{Zitt} \\ Institut de Math\'ematiques de Bourgogne, Universit\'e de Bourgogne, \\
9 avenue Alain Savary, 21078 Dijon Cédex, France \\
email: \{Herve.Cardot, Peggy.Cenac,  Pierre-Andre.Zitt\}@u-bourgogne.fr
} 
\maketitle
\begin{abstract}
A recursive estimator of the conditional geometric median in Hilbert spaces is studied. It is based on a stochastic gradient algorithm whose aim is to minimize a weighted $L_1$ criterion and is consequently well adapted for robust online estimation. The weights are controlled by a kernel function and an associated bandwidth. Almost sure convergence and $L^2$ rates of convergence are proved under general conditions on the conditional distribution as well as the sequence of  descent steps of the algorithm and the sequence of bandwidths.
Asymptotic normality is also proved for the averaged version of the algorithm with an optimal rate of convergence. 
A simulation study confirms the interest of this new and fast algorithm when the sample sizes are  large. Finally, the ability of these recursive algorithms to deal with very high-dimensional data is illustrated on the robust estimation of television audience profiles conditional on the total time spent watching television over a period of 24 hours.
\end{abstract}

\section{Introduction}

It is not unusual nowadays to get large samples of high-dimensional or functional data together with real covariates that are correlated with the functional variable under study.  The estimation of how  the shape of the functional response may depend on real or functional covariates has been deeply studied in the statistical literature : linear models for functional response have been proposed by \cite{Faraway97}, \cite{CFF02} or \cite{Bosq} (see also \cite{RamsaySilverman2005}) and \cite{GCCR2010} whereas nonlinear relationships are studied in \cite{lecoutre90}, \cite{CMW2004}, \cite{Lian2007}, \cite{Cardot2007}, \cite{Lian2011} and \cite{FLTV2011}.

The main drawback of  all the above mentioned estimators, whose target is the conditional expectation, is that they all rely, explicitly or not, on least squares and are consequently sensitive to outliers. 
In such a context of large samples of high dimensional data, outlying observations, which  may not be uncommon, might be hard to detect with automatic procedures. 
 Directly considering  robust indicators of centrality  such as medians is a way to deal with this issue. If $Y$ be a random variable taking values in a Hilbert space $H,$ its geometric median $m$ (also called spatial median or $L_1$-median,  see \cite{Sma90}
 for a survey) is defined as follows
\begin{align}
m:=\argmin_{\alpha \in H} \EE{\nrm{Y-\alpha} - \nrm{Y}}.
\label{def:Hmed}
\end{align}
The median $m$ is  uniquely defined under simple conditions when the dimension of $H$ is larger than or equal to 2, it has a 0.5 breakdown point (\cite{Kem87}) as well as a bounded gross sensitivity error (\cite{CCZ11}).
When one has a sample at hand, algorithms based on the minimization of the empirical version of risk \eqref{def:Hmed} have been proposed by \cite{VZ00} and properties of such robust estimators can be found in the recent review by \cite{MNO2010}. Nevertheless, these computational techniques may not be able to handle  very large samples of high-dimensional data since they require to store all the data. An  alternative approach, developed by \cite{ChaouchGoga2012} and which can cope with this issue, consists in considering unequal probability sampling techniques in order to select, in a effective way, subsamples with sizes much smaller than the initial sample size. 

\medskip

We suggest in this work another direction based on  recursive techniques   which do not require to store all the data. Another interest of these recursive approaches is that they allow automatic update of the estimators if, for example, the data arrive sequentially. Recently, a simple recursive algorithm which gives  efficient estimates of the geometric median in separable Hilbert spaces has been proposed by \cite{CCZ11}. It is shown that averaged versions of classic stochastic gradient algorithms have a limiting normal distribution that is the same as the distribution of the static estimator based on a direct minimization of the empirical version of risk \eqref{def:Hmed}.

\medskip

In a finite dimension context, \cite{CG2000} and  \cite{ChengDeGooijer2007} proposed to introduce a kernel function $K$ in the empirical version of \eqref{def:Hmed} in order to take covariate effects into account. The kernel weights are controlled by a sequence of bandwidth values  that tends to zero when the sample size increases in order to build consistent estimates of  the conditional geometric median.
With the same ideas of local approximation of the conditional distribution, we study, in this work, a modification of the recursive algorithm suggested in \cite{CCZ11}. It consists in introducing weights, controlled by a kernel function,  in order to build consistent recursive estimators of the conditional geometric median. The response variable is also allowed to take values in a separable Hilbert space. For real response, recursive estimators of the regression function based on kernel weights have been introduced by \cite{Revesz1977}  whereas a deep  study of their asymptotic properties, which also includes averaged estimation procedures, is proposed in \cite{MPS2009}.

\medskip

The paper is organized as follows. In Section 2, we first define the stochastic gradient  recursive estimator as well as its averaged version for the case of a real covariate. Note that our results could be extended  to multidimensional covariates. We state the asymptotic normality, under general conditions, of the averaged algorithm in separable Hilbert spaces, with an optimal rate of convergence.
The regularity hypotheses, which are much weaker than those of \cite{CG2000}, are also expressed in terms of the Wasserstein distance between the conditional distributions.

In Section 3, a comparison of the static approach, which consists in minimizing the empirical version of risk  \eqref{def:Hmed}, with the stochastic gradient estimator and its averaged version is performed on a  simulation study. It confirms the good behavior as well as the stability, with respect to  the descent steps, of the averaged algorithm. The ability of this estimator to deal with large samples of very high-dimensional data is then illustrated on the estimation of television audience profiles given the total time spent watching television. Proofs are gathered in Section 4.

\section{Notations, hypotheses and main results}

Let $(Y,X)$ be a pair of random variables taking values in $H \times \R,$ where $H$ is a Hilbert space whose norm is denoted by $\nrm{\cdot}.$
Suppose that $X$ is continuous, and denote by $p(x)$ its density at $x\in \xR$. For any $x$ in the support of $X$, denote by $\mu_x$ the 
conditional law of $Y$ given $X = x$. 
Consider, for $(\alpha,x) \in H \times \R,$ the following functional
\begin{align}
  \label{eq=defGalpha}
G(\alpha,x) &\egaldef  p(x) \ \EE{\nrm{Y-\alpha} - \nrm{Y} | X=x}.
\end{align}

The geometric median of $Y$ given $X=x$, denoted by $m(x)$, is defined as the solution of the following optimization problem:
\begin{align}
m(x):=\argmin_{\alpha \in H} \ G(\alpha,x) .
\label{def:condHmed}
\end{align}
The solution of \eqref{def:condHmed} is unique provided that the conditional distribution $\mu_x$ is not supported by a straight line (\cite{Kem87}). We suppose
from now on the following assumption.
\begin{itemize}
  \item [\textbf{A1.}]  For every $x$ in the support of the probability density function $p$ of the random variable $X$, 
    $\mu_x$ is not concentrated on a straight line: 
  for all $v\in H,$  there is $w \in H$ such that $\scal{v,w} = 0$ and
  \begin{equation}
    \Var ( \scal{w,Y}|X=x ) > 0.  
    \label{eq:supportCdtn}
  \end{equation}
\end{itemize}

Suppose we have a sequence $(X_n,Y_n)_{n\geq 1}$ of independent copies of $(X,Y)$. 
In the unconditional case where the $X$ variable is not taken into account, one can look for 
the unconditional median, 
\emph{i.e.} the minimum $m$ defined by (\ref{def:Hmed}). Under weak hypotheses, the median is uniquely defined as the zero of the derivative:
\[-\EE{\frac{Y-\alpha}{\nrm{Y-\alpha}}}.\]

 We introduced in \cite{CCZ11} the following recursive estimator of $m$:
\begin{equation}
  Z_{n+1} = Z_n + \gamma_n \frac{Y_{n+1} - Z_n}{\nrm{Y_{n+1} - Z_n}},
  \label{eq=defOldAlgo}
\end{equation}
where $\gamma_n$ was a well-chosen deterministic sequence. In the present case, the law of $Y_n$ is not the  conditional law $\mu_x$, so 
this idea does not work directly. However, it is natural to see $Y_n$ as an approximate sample of $\mu_x$ if $X_n$ happens to be 
very close to $x$. 
Therefore, a simple estimator can be built by introducing weights, through a kernel function $K$, whose properties will be specified later.
We modify \eqref{eq=defOldAlgo} as follows to take the weights into account, and define our recursive estimator of $m(x)$:
\begin{align}
Z_{n+1}(x) &= Z_n(x) + \gamma_n \frac{Y_{n+1} - Z_n(x)}{\nrm{Y_{n+1} - Z_n(x)}} \frac{1}{h_n} K \left( \frac{X_{n+1}-x}{h_n}\right)
\label{def:algo1}
\end{align}
with two deterministic sequences of tuning parameters $h_n$ and $\gamma_n$ whose properties are given below. 

\medskip
For a constant sequence $(h_n)$, this algorithm converges towards the minimum of the modified objective function:
\begin{equation}
G_h(\alpha,x) \egaldef  \EE{ \left(\nrm{Y-\alpha} - \nrm{Y}\right) \frac{1}{h}K \left( \frac{X-x}{h} \right) }.
\label{def:weightHmed}
\end{equation}
The partial derivative of $G_h$ with respect to $\alpha$ is an element of $H$ defined by
\begin{align}
\Phi_h(\alpha) &\egaldef   \nabla_\alpha G_h(\alpha,x) \nonumber \\
 &= - \EE{\frac{ Y - \alpha}{ \nrm{ Y- \alpha}}  \frac{1}{h} K \left( \frac{X-x}{h}\right)}.
\label{def:PhiH}  
\end{align}

We will see in Proposition \ref{prp=PhiEtPhiH} that, under suitable hypotheses, when $h$ goes to zero,  $\Phi_h$ goes to the gradient $\Phi$ of $G$, defined by:
\begin{equation}
\Phi(x,\alpha) =  - p(x) \ECond{\frac{Y - \alpha}{\nrm{Y-\alpha}}}{X=x}.
  \label{eq=defPhi}
\end{equation}

The idea of using a kernel, and of assigning a large weight to $Y_n$ when $X_n$ is close to $x$ can only work if the conditional law $\mu_{x'}$ varies, in some sense, regularly. A natural way of expressing this regularity is through the  Mallows-Wasserstein distance. 
  Let us recall its definition.

\begin{defn}
  Let $\mu$ and $\nu$ be two probability measures on $H$ with finite second order moments. Let $\mathcal{C}$ be the
  set of couplings of $\mu$ and $\nu$, \emph{i.e.} the set of measures $\pi$ on $H\times H$
  whose first marginal is $\mu$ and whose second marginal is $\nu$. 

  The Wasserstein distance between $\mu$ and $\nu$ is given by:
  \[
    \wass{\mu,\nu} = \left( \inf_{\pi \in \mathcal{C}} \int \nrm{x-y}^2 d\pi(x,y) \right)^{1/2}.
    \]
\end{defn}

We may now state our assumptions. 
\begin{itemize}
  \item [\textbf{A2.}] The probability density function $p$ of the random variable $X$ is bounded and satisfies a uniform H\"older condition : 
there are two constants $\beta>0$  and $C_2>0$ such that
  \begin{align*}
	     \forall (x,x')\in \xR^2,  \quad
     |p(x) - p(x')|  \leq C_2 |x-x'|^{\beta}.
   \end{align*}

We denote by $p_{\max} = \sup_{x \in \R} p(x).$ 

     \item [\textbf{A3.}] The gradient $\Phi(x,\alpha)$ defined by \eqref{eq=defPhi}
  satisfies a uniform Hölder condition with coefficient $\beta.$ There is $\cstHolder>0$ such that
  \begin{equation}
    \label{eq=A3}
	     \forall (x,x')\in \xR^2, \forall \alpha \in H, \quad
     \nrm{\Phi(\alpha,x) - \Phi(\alpha,x')} \leq \cstHolder |x-x'|^{\beta}.
   \end{equation}
 
   \item [\textbf{A4.}] The conditional law $\mu_x = \law(Y|X = x)$ varies regularly with $x$: there are two constants 
     $\cstWass$ and $\beta$ such that 
     \begin{equation}
       \label{eq=assumptionWass}
       \wass{\mu_x, \mu_{x'}} \leq \cstWass\abs{x-x'}^\beta.
     \end{equation}
   \item[\textbf{A5.}] The kernel function $K$ is positive, 
   bounded with compact support and satisfies
  \begin{align*}
\int_\R K(u) du &= 1. 
\end{align*}
  \item [\textbf{A6.}] 
   There is a constant $\cstMoment$ such that:
       \begin{equation}
     \forall \alpha \in H, \forall x, \quad
     \EE{\nrm{Y-\alpha}^{-2} | X=x} \leq \cstMoment. 
    \label{eq=boundednessOnBall}
  \end{equation}
\end{itemize}
\begin{rem}

  Without loss of generality, we suppose that the constant $\beta$ in \assm{2}, \assm{3} and \assm{4} has always the same value. 

  Assumption \assm{3} is a regularity assumption that is required to control the approximation error and to prove the convergence of the algorithm.
  Assumption \assm{4}
  seems to be more natural, and we prove in section \ref{sec=wasserstein} that, together with \assm{6}, it implies \assm{3}.

  Hypotheses \assm{2} and \assm{5} are classical in nonparametric estimation and could be weakened at the expense of more complicated proofs. For classical properties of kernel estimators under general hypotheses, see for example \cite{WandJones1995}.
  
  Similarly, Hypothesis \assm{6} is stated quite strongly here, in order to avoid additional technicalities in the proof of the asymptotic normality if the averaged algorithm. 
  See \cite{CCZ11} for a relaxed version, under which the same results should hold. Informally 
  it forces the law to be ``spread out'' and this avoids pathological behaviors of the algorithm. 
\end{rem}

We have three main results. The first one states  the almost sure convergence of the algorithm. 

\begin{thm} Under assumptions \assm{1}--\assm{3} and \assm{5},
and if $\sum_n \gamma_n=\infty$, $\sum_n \gamma_n^2 h_n^{-1} < \infty$ as well as $\sum_n \gamma_n h_n^{\beta} < \infty,$ then, for all $x$ such that $p(x)>0,$
\begin{align*}
\lim_{n \rightarrow \infty} \nrm{Z_n(x) - m(x)} &= 0 \quad a.s.
\end{align*}
\label{thm:cvps}
\end{thm}
\begin{rem}
In the following, for simplicity, we choose the step size and window size as inverse powers of $n$:
\begin{align}
  \gamma_n &= \frac{c_\gamma}{n^\gamma}, & 
  h_n      &= \frac{c_h}{n^h}. 
  \label{eq=stepSize}
\end{align}
With these choices the assumptions on the step sizes are:
\begin{align}
  \gamma &\leq 1,    &
  2 \gamma - h &>1,  &
  \gamma + \beta h &> 1.
  \label{eq=contraintesHBetaGamma}
\end{align}
\end{rem}

The assumptions on $h$ and $\gamma$ are always satisfied if we choose $\gamma = 1$ and $h<1$. 
However, as shown in the simulation study, the performances of  algorithm \eqref{def:algo1} strongly depend on the choice 
of the steps $\gamma_n$ and particularly on the constant $c_\gamma$. 
Therefore, we also introduce the following averaged algorithm which is less sensitive to the choice of the step sizes $\gamma_n$ and has nice convergence properties,

\begin{align}
  \overline{Z}_{n+1}(x) &= \frac{1}{n} \sum_{k=1}^n Z_k(x).
  \label{def:averaged}
\end{align}
Our main result is a central limit theorem on this averaged algorithm. 
To adapt the proof of the corresponding CLT from \cite{CCZ11}, we need a good \emph{a priori} bound on the error $Z_n(x) - m$. 

\begin{prop}
  \label{prp=vitesseQuadratique}
  Suppose that $x$ is such that $p(x)>0$ and that $\gamma\leq 1$, $2\gamma-h>1$, $\gamma+\beta h>1$, 
  and $h(1+2 \beta)\geq \gamma.$ 
  Under Assumptions \assm{1}--\assm{3} and \assm{5},
  there exist an increasing sequence of events $(\Omega_N)_{N\in\xN}$, and constants $C_N$, such that $\Omega = \bigcup_{N\in\xN} \Omega_N$, and
  \begin{equation}
    \forall N, \quad 
    \EE{\ind{\Omega_N} \nrm{Z_n - m(x)}^2}
    \leq C_N \frac{\ln(n)}{n^{\gamma-h}} . \nonumber
  \end{equation}
\end{prop}

This proposition tells us that, up to a logarithmic factor, the optimal rates of convergence in nonparametric estimation can be attained for well chosen values of the parameter $\gamma$ and $h$. If $\gamma=1$ and $h=(1+2\beta)^{-1},$ then,
\begin{align}
\nrm{Z_n - m(x)}^2 &= O_p \left( \ln(n) \ n^{-2\beta/(2\beta +1)} \right).
\end{align}
Finally our main result is the following central limit theorem for the averaged algorithm. 

\begin{thm}
  \label{thm=CLTavg}
  Assume  \assm{1}, \assm{2} and  \assm{4}--\assm{6}. 
   Let $x$
  satisfy $p(x)>0$. 
  If $\gamma<1$, $2\gamma-h>1$, $\gamma+\beta h>1$ and $h>(2\beta+1)^{-1}$, then:
\[
\frac{n}{\sqrt{\sum_{k=1}^n\frac{1}{h_k}}}\left( \Zbar_n - m(x) \right)
  \cvl
  \mathcal{N} \left(0, \Gamma^{-1}\covLimite \Gamma^{-1}\right),
\]
where
\begin{align}
\label{eq=defCovLimite}
\covLimite 
&=
p(x) \left(\int K^2(u)du \right) \EE{\frac{(Y-m(x))\otimes(Y-m(x))}{\nrm{Y-m(x)}^2} \middle| X = x}, \\
\Gamma &= \EE{\frac{1}{\nrm{Y - m(x)}}\left( \idt - \frac{ (Y - m(x)) \otimes (Y - m(x))}{\nrm{Y - m(x)}^2} \right) \middle| X = x}.
\end{align}
 \end{thm}

As  shown in \cite{CCZ11} in the unconditional framework, the operator $\Gamma$ has a bounded inverse under assumption \assm{1}, so that the asymptotic variance operator is well defined.
Let us also remark that with our assumptions on the sequence of bandwidths, we have
\begin{equation}
\label{equiv}
\sum_{k=1}^n \frac{1}{h_k}= \frac{n}{h_n}.\frac{1}{1+h}+o\left(nh_n^{-1}\right).
\end{equation}
Consequently, the rate of convergence in the CLT is of order $\sqrt{nh_n},$ which is the usual rate of convergence in distribution for nonparametric regression, provided that the bias term is negligible compared to the variance. This latter condition is  ensured by the  additional condition $h < (2 \beta +1)^{-1}$ and we have, with Theorem  \ref{thm=CLTavg},
\[
\sqrt{n h_n}\left( \Zbar_n - m(x) \right)
  \cvl
  \mathcal{N} \left(0, \frac{1}{1+h} \Gamma^{-1}\covLimite \Gamma^{-1}\right).
\]
As in the real regression case (see  \cite{MPS2009}) it turns out that the averaged estimator  has a smaller asymptotic variance, with in our case  a factor $(1+h)^{-1},$ than the classical kernel estimator which minimizes the empirical version of risk \eqref{def:weightHmed}.

\begin{figure}
\begin{tikzpicture}[scale = 5,line label/.style={near start,sloped,above,font=\small}]
  \draw[help lines,->] (0,0) -- (1.2,0) node [below right,black] {$\gamma$};
  \draw[help lines,->] (0,0) -- (0,1.2) node [above left,black]  {$h$}     ;
  \draw[help lines] (0,1) -- (1,1) -- (1,0);
  \draw[name path=montant,very thin]    (0,0) -- node[line label] {$h(1+2\beta)\geq \gamma$} (1,0.33) node[coordinate,label=right:{$1/3$}] (d) {D};
  \draw[name path=constant,very thin]   (0,0.33) -- node[line label] {$h(1+2\beta) \geq 1$} (1,0.33);
  \draw[name path=descendant,very thin] (0,1) node[left] {$1$} -- node[line label] {$\gamma + \beta h >1$} (1,0) node [below] {$1$};
  \draw[name path=supermontant, very thin] (0.5,0) -- node[line label, near end] {$2\gamma - h > 1$} (1,1) node[coordinate] (c) {C};
  \draw[name intersections={of=montant and descendant, by= {a}}];
  \draw[name intersections={of=supermontant and descendant, by= {b}}];
  \fill[fill=black!10] (1,0) -- (b) -- (c) -- (d) -- cycle;
  \fill[fill=black!35] (a) -- (b) -- (c) -- (d) -- cycle;
  \fill[fill=black!55] (b)  --(c) --(d)  --cycle;
  \draw [dashed, thick] (b) -- (c) -- (1,0)--cycle;
  \draw [dashed]  (b)--(1,0);
  \draw (a)--(d);
  \draw[help lines] (b) |- node[black,midway,below] {$2/3$} (0,0) ;
  \node[rectangle,draw] at (0.5,1.2) {$\beta = 1$};
\end{tikzpicture}
\hfill
\begin{tikzpicture}[scale = 5,line label/.style={near start,sloped,above,font=\small}]
  \draw[help lines,->] (0,0) -- (1.2,0) node [below right,black] {$\gamma$};
  \draw[help lines,->] (0,0) -- (0,1.2) node [above left,black]  {$h$}     ;
  \draw[help lines] (0,1) -- (1,1) -- (1,0);
  \draw[name path=montant,very thin] (0,0) -- node[line label] {$h(1+2\beta)\geq \gamma$} (1,0.5) node[coordinate,label=right:{$1/2$}] (d) {D};
  \draw[name path=constant,very thin] (0,0.5) -- node[line label] {$h(1+2\beta) \geq 1$} (1,0.5);
  \draw[name path=descendant, very thin] (0.5,1) -- node[line label,below] {$\gamma + \beta h >1$} (1,0) node [below] {$1$};
  \draw[name path=supermontant, very thin] (0.5,0) -- node[line label, near end] {$2\gamma - h > 1$} (1,1) node[coordinate] (c) {C};
  \draw[name intersections={of=montant and descendant, by= {a}}];
  \draw[name intersections={of=supermontant and descendant, by= {b}}];
  \fill[fill=black!10] (1,0) -- (b) -- (c) -- (d) -- cycle;
  \fill[fill=black!35] (a) -- (b) -- (c) -- (d) -- cycle;
  \fill[fill=black!55] (b)  --(c) --(d)  --cycle;
  \draw [dashed, thick] (b) -- (c) -- (1,0)--cycle;
  \draw [dashed]  (b)--(1,0);
  \draw (a)--(d);
  \node[left] at (0,1) {$1$};
  \draw[help lines] (b) |- node[black,midway,below] {$3/4$} (0,0) ;
  \node[rectangle,draw] at (0.5,1.2) {$\beta = 1/2$};

\end{tikzpicture}

{\small
In this picture we represent the possible choices for the parameters $h$ and $\gamma$, when $\beta$ varies. 
On the left is the most regular case where $\beta = 1$, on the right we set $\beta = 1/2$. 
In both cases, if $(\gamma,h)$ lies in the lighter region, Theorem~\ref{thm:cvps} holds and the algorithm converges. 
In the middle region, the algorithm converges and the additional convergence estimate
of Proposition~\ref{prp=vitesseQuadratique} holds. Finally, if $(\gamma,h)$ is in the darker region, the CLT of Theorem~\ref{thm=CLTavg} holds. 
All these two regions get smaller when 
$\beta$ is small. Note that even in the most regular case $\beta = 1$, in order to fulfill the hypotheses of Theorem~\ref{thm=CLTavg}, it is necessary to choose $\gamma$ larger than $2/3$ 
and $h$ larger than $1/3$.
}
\caption{Possible choices for $h$ and $\gamma$.}
\end{figure}

\begin{rem}
Proceeding exactly as in the proof of Theorem~\ref{thm=CLTavg}, it is possible to establish a CLT for another weighted version of the algorithm $\widetilde{Z_n}=\frac{1}{n}\sum_{k=1}^n\sqrt{h_k}(Z_k-m)$, which is the empirical mean of $\sqrt{h_n}(Z_n-m)$. Under the same assumptions of Theorem~\ref{thm=CLTavg}, one has:
\[
  \sqrt{n}\widetilde{Z_n}
  \cvl
  \mathcal{N} \left(0, \Gamma^{-1}\covLimite \Gamma^{-1}\right).
\]
\end{rem}
\section{Examples}
We first consider a simple simulated example in order to compare the performances of the averaged algorithm with the more classic static one as well as the recursive Robbins-Monro estimator without averaging. Then, the ability of our recursive averaged estimator to deal with large samples  of very high-dimensional data is illustrated  on the robust estimation of television audience profiles, measured at a minute scale over a period of 24 hours, given the total time spent watching television. All functions are coded in \Rlogo \ (\cite{R10}) and are available on request to the authors.

\subsection{A simulated example}

Consider a Brownian motion $Y$ measured at $d$ equispaced time points in the interval $[0,1],$ so that we have
$\mathbf{Y} = (Y(t_1), \ldots, Y(t_d)).$ Besides, suppose that we know the mean value $X=\int_0^1 Y(t) dt$ of each trajectory $Y.$  We can look for the conditional (geometric) median of vector $\mathbf{Y}$ given $X.$
The joint distribution of $(\mathbf{Y},X)$ is clearly Gaussian with $\E \mathbf{Y} = 0,$ $\E{X}=0,$ 
\[\Cov(Y(t_j),Y(t_\ell)) = \min(t_j, t_\ell), \quad  \Var(X) = \frac 13 \quad \mbox{and  } \Cov(X,Y(t_j)) = t_j\left(1-\frac{t_j}{2}\right).\]
Consequently, the distribution of $\mathbf{Y}$ given $X=x$ is Gaussian with  conditional expectation, for $j=1, \ldots, p,$ 
\begin{align*}
\ECond{Y(t_j)}{X=x} &= \frac{3}{2} t_j (2 - t_j) x,
\end{align*}
and a covariance matrix that does not depend on $x.$
By symmetry of the Gaussian distribution, it is also clear that the conditional expectation is equal to the conditional geometric median, when $H = \mathbb{R}^d$ equipped with the usual Euclidean norm, so that
\begin{align}
m(t_j,x) &=  \frac{3}{2} t_j (2 - t_j) x .
\label{def:condMsim}
\end{align}
The hypotheses on the density $p$ are clearly satisfied since $X$ is a Gaussian random variable.
Furthermore, the Wasserstein distance between two Gaussian laws with expectations $m_1$ and $m_2$ and  the same covariance matrix is simply $\nrm{m_1 - m_2}$, (see \textit{e.g.} \cite{GivensShortt1984}) so that we can deduce, with \eqref{def:condMsim}, that $\beta=1$ in Assumption~\assm{4}.

\medskip

We draw $n$ i.i.d.\ copies of $(\mathbf{Y},X)$ and we focus in this simulation study on the geometric median of $Y$ given $x=0.39,$ which corresponds to the value of the third quartile of $X.$ Note that our conclusions remain unchanged for other non extreme values of $X$.

We first compute the static estimator, named "static kernel" in the following. It is  based on a direct minimization, with the Weiszfeld's algorithm (see \cite{VZ00} and \cite{MNO2010}), of  
\begin{align}
\boldsymbol{\alpha} & \mapsto \sum_{i=1}^n w_i \nrm{\mathbf{Y}_i - \boldsymbol{\alpha}},
\label{def:emprisk}
\end{align}
 where $w_i = \left[ \sum_{\ell=1}^n K(h_n^{-1} (X_\ell - x))\right]^{-1}K(h_n^{-1} (X_i - x))$ and $K$ is the Gaussian kernel. 

The Robbins Monro estimator $Z_n,$ defined in \eqref{def:algo1}, and the averaged estimator $\overline{Z}_n,$ defined in \eqref{def:averaged}, are run for 10 starting points chosen randomly in the sample. Among the 10 estimations,  we retain the one with the smallest empirical risk \eqref{def:emprisk}.

The accuracy  of the different estimators $\widehat{m}$ 
are compared, for different values of the bandwidth $h$ and sample sizes $n,$ with the quadratic criterion,
\begin{align}
R(\widehat{m}) &= \frac{1}{d} \sum_{j=1}^d \left( m(t_j) - \widehat{m}(t_j) \right)^2.
\label{def:riskL2}
\end{align}
 Since $\beta=1,$ we can choose  $\gamma=9/10$ and $h=3/10,$ $c_h=1$, so that the quadratic estimation error for the Robbins-Monro algorithm, will be, up to the $\ln(n)$ factor, of order  $n^{-6/10}$ (see Proposition \ref{prp=vitesseQuadratique}).

Note that, for simplicity of comparison with the static kernel estimator, we also consider fixed values for $h_n  \in \{0.05, 0.10, 0.15, 0.20, 0.25\}$ and  take in this case $\gamma=2/3.$ We are aware that the assumptions needed for the asymptotic convergence are not satisfied but the sample size is fixed in advance here.

We first present in Table~\ref{tab:n500} the mean value, over 500 replications,  of the MSE  defined in \eqref{def:riskL2}, when estimating the conditional median with a sample size of $n=500$ in dimension $d=100.$  For comparison and interpretability of the results, note that $100 R(0)= 18.4.$

We note that, when the sample size is moderate (\textit{i.e.} $n=500$), the interest of considering the averaged recursive estimation procedure is less evident than in the unconditional case (see \cite{CardCC10}) since the Robbins-Monro estimator $Z_n$ defined in \eqref{def:algo1} can perform, for well chosen values of the tuning parameters $c_\gamma$ and $h_n,$ nearly as well as the static estimator. Nevertheless, we can remark that $Z_n$ is highly sensitive to the values of the tuning parameters and its performances deteriorate much with small variations of these parameters as seen in Table~\ref{tab:n500}.
This is not the case of the averaged estimator $\overline{Z}_n,$ defined in \eqref{def:averaged}, which is much less sensitive and thus allows less sharp choices of the values of the tuning parameters provided the descent steps do not force the algorithm to converge too rapidly.
We note again (see \cite{CardCC10}) that for too small values of $c_\gamma$ (i.e.\ $c_\gamma = 0.1$), the algorithm converges too quickly and  averaging leads to estimations that are   outperformed by the direct Robbins-Monro approach. A way to deal with this drawback is to perform  averaging only after a certain number of iterations. All these remarks are clearly illustrated in Figure~\ref{fig:compRMAVE} which presents the estimation error, defined in \eqref{def:riskL2}, for both algorithms and  for different values of~$c_\gamma.$ 

 When the sample size gets larger the interest of the averaging step becomes clearer since  the estimation error of the Robbins-Monro estimator are always larger as soon as $c_\gamma \geq 1$ (see Table~\ref{tab:n2000}). Furthermore, the estimation errors of the static kernel estimator and the averaged recursive one are also now very close to each other. 

\begin{table}[htdp]
\caption{Mean estimation errors ($\times$ 100) of the different estimators, for $n=500,$ $d=100,$  and descent parameter $\gamma=2/3$ when $h_n$ has a constant value and $\gamma=0.9$ when $h_n = n^{-h}$ with $h=\gamma/3 = 0.3.$}
\begin{center}
\begin{tabular}{|c|c|c|c|c|c|c|} \hline
 & \multicolumn{6}{|c|}{Bandwidth $h_n$} \\  \hline
 &  0.05 & 0.10 & 0.15 & 0.20 & 0.25 &  $n^{-0.3}$\\ \hline \hline 
Static kernel & 0.349 & 0.179 & 0.148 & 0.172 & 0.245& \\ \hline
Robbins Monro & & & & & &\\
$c_\gamma = 0.1$ & 0.689& 0.625& 0.659& 0.769 & 0.912 & 2.458\\
$c_\gamma = 0.3$ & 0.370& 0.194& 0.159& 0.178& 0.253 & 0.332 \\
$c_\gamma = 1$ & 0.590 & 0.297& 0.229& 0.240 & 0.297 & 0.183\\
$c_\gamma = 3 $ & 1.177& 0.647& 0.486& 0.425 & 0.453 & 0.248 \\ \hline 
Averaged  & & & & &  &\\
$c_\gamma = 0.1$ & 1.047& 1.000 & 1.051& 1.160& 1.336 & 2.995\\
$c_\gamma = 0.3$ & 0.406 & 0.213& 0.178&0.202  & 0.287 & 0.534 \\
$c_\gamma = 1$ & 0.402& 0.195& 0.160& 0.182& 0.252 & 0.192\\
$c_\gamma = 3 $ & 0.443 & 0.209 & 0.163 & 0.252& 0.256 & 0.170\\ \hline 
\end{tabular}
\end{center}
\label{tab:n500}
\end{table}%

\begin{table}[htdp]
\caption{Mean estimation errors ($\times$ 100) of the different estimators, for $n=2000,$ $d=100,$  and descent parameter $\gamma=2/3$ when $h_n$ has a constant value and $\gamma=0.9$ when $h_n = n^{-h}$ with $h=\gamma/3 = 0.3.$}
\begin{center}
\begin{tabular}{|c|c|c|c|c|c|c|} \hline
 & \multicolumn{6}{|c|}{Bandwidth $h_n$} \\  \hline
 &  0.05 & 0.10 & 0.15 & 0.20 & 0.25 & $n^{-0.3}$  \\ \hline \hline 
Static kernel & 0.082 & 0.053 & 0.060 & 0.099 & 0.176 &  \\ \hline
Robbins Monro & & & & & & \\
$c_\gamma = 0.1$ & 0.139& 0.128& 0.149& 0.205 & 0.324 & 1.321\\
$c_\gamma = 0.3$ & 0.095& 0.061& 0.065& 0.103& 0.181& 0.083 \\
$c_\gamma = 1$ & 0.173 & 0.104& 0.098& 0.126 & 0.194 & 0.061\\
$c_\gamma = 3 $ & 0.403 & 0.230& 0.175& 0.192 & 0.253 & 0.096 \\ \hline 
Averaged  & & & & & & \\
$c_\gamma = 0.1$ & 0.240& 0.237 & 0.270& 0.332& 0.484 & 1.712\\
$c_\gamma = 0.3$ & 0.091 & 0.058& 0.065&0.102  & 0.183 & 0.138 \\
$c_\gamma = 1$ & 0.090& 0.057& 0.063& 0.101& 0.178 & 0.060\\
$c_\gamma = 3 $ & 0.097 & 0.058 & 0.064 & 0.101& 0.180 & 0.057\\ \hline 
\end{tabular}
\end{center}
\label{tab:n2000}
\end{table}%

  \begin{figure}
   \begin{center}
  \includegraphics[height=15cm]{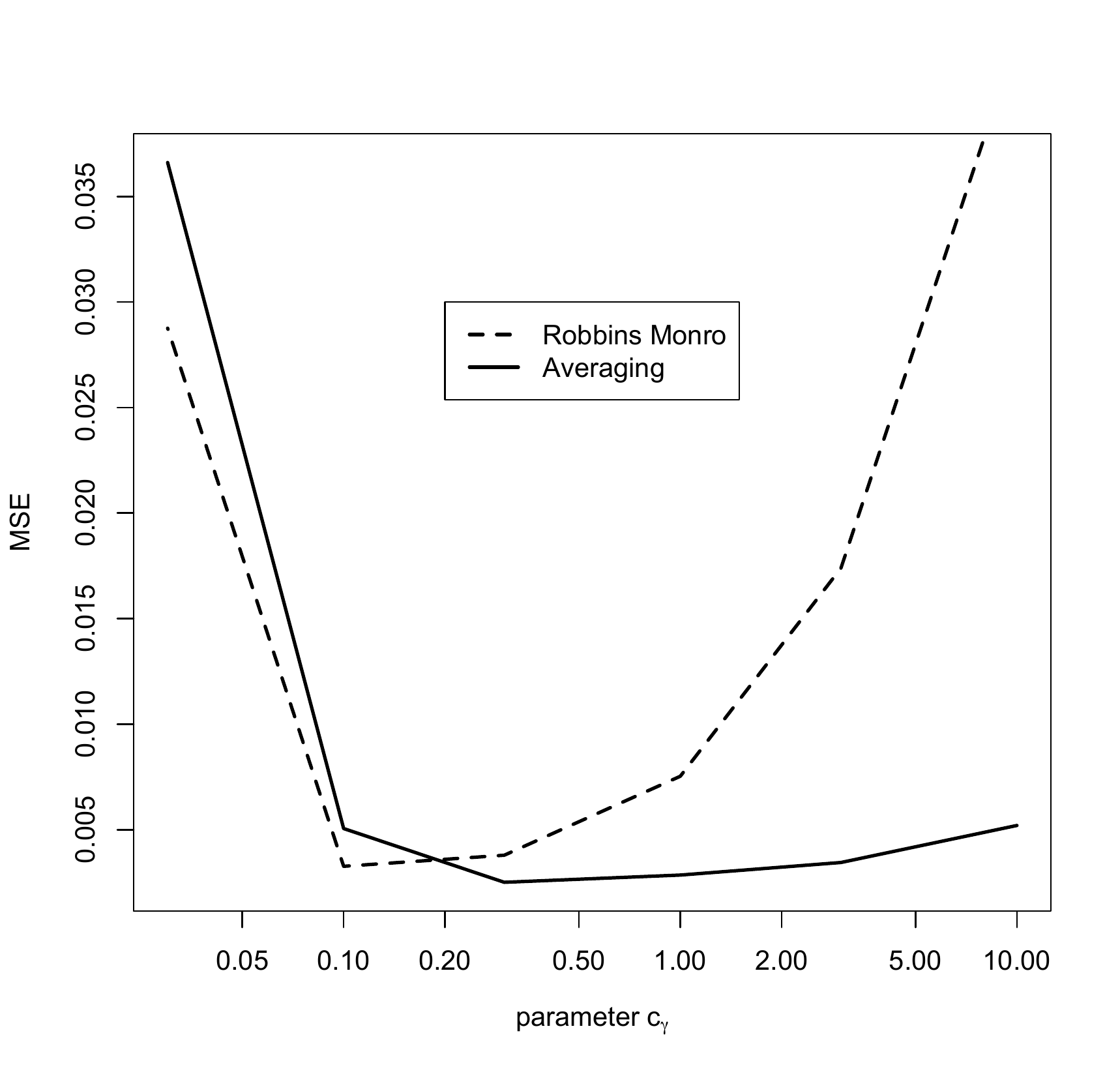}
 \caption{Comparison of the two recursive algorithms according to the mean square error of estimation for different values of $c_\gamma$ (with a logarithmic scale). The sample size is $n=500$ and $d=100.$}
 \label{fig:compRMAVE}
 \end{center}
   \end{figure}

\subsection{Television audience data}

We have a sample of  $n=5422$ individual audiences measured every minute over a period of 24 hours  and by the M\'ediam\'etrie company in France. For $j=1, \ldots, 1440,$ an observation $Y_i(t_j)$ represents the proportion of time spent by the individual $i$ watching television during the $j$\textsuperscript{th} minute of this day. 
Thus, each vector $\mathbf{Y}_i$ belongs to $[0,1]^{1440}.$ Note that in fact the first measurement $t_1$ is made at 3 AM  of day $d$ and the last one just before 3 AM of day $d+1$ (see Figure \ref{fig1}). A more detailed description of these data can be found in \cite{CCZ11}.

We are interested in estimating  television consumption behaviors, over a 24 hours period, according to the total time spent watching television.
The covariate  $X,$ is the proportion of time spent watching television over the considered period, $X_i = (\sum_j Y_i(t_j))/1440$,  for $i=1, \ldots, n=5422.$
We consider  the  quantile values of  $X$ which are, in the sample, $q_{25} = 0.0599,$ $q_{50} = 0.128,$ $q_{75} = 0.225$ and $q_{90} = 0.348.$
This means for example, that the ten percent of consumers with the highest consumption levels  spend more than 34.8 \% of their time watching television whereas the 25 \% of consumers with the lowest consumption levels  spend less than 6\% of their time watching television.

We have drawn in Figure~\ref{fig1} the estimated conditional median profiles with a bandwidth value set to $h_n=0.05$ and a descent parameter $c_\gamma=0.5,$ for $x \in \{q_{25}, q_{50}, q_{75},q_{90}\}.$ For comparison and better interpretation, we have also plotted the overall geometric median as well as the mean profile.  One can note that the shape of the conditional profiles strongly depend on the value of the covariate and that multiplicative models  that could be thought to be natural (see the simulation study), are in fact not adapted for modeling the conditional audience median profiles. This is clear if we compare, for example, the levels of the conditional median curves for $x= q_{75}$ and $x=q_{90}$ at time 15 and at time 21. Around 21, their values are approximately the same and are close to the global maximum whereas  at time 15  the value of the conditional median for $x=q_{90}$ is about twice the value of the conditional median for $x=q_{75}.$

From a computational speed point of view, for one starting point, our algorithm, which takes less than two seconds, is about 70 times faster than the static estimator which requires 140 seconds to converge.

  \begin{figure}
   \begin{center}
  \includegraphics[height=14cm,width=15.5cm]{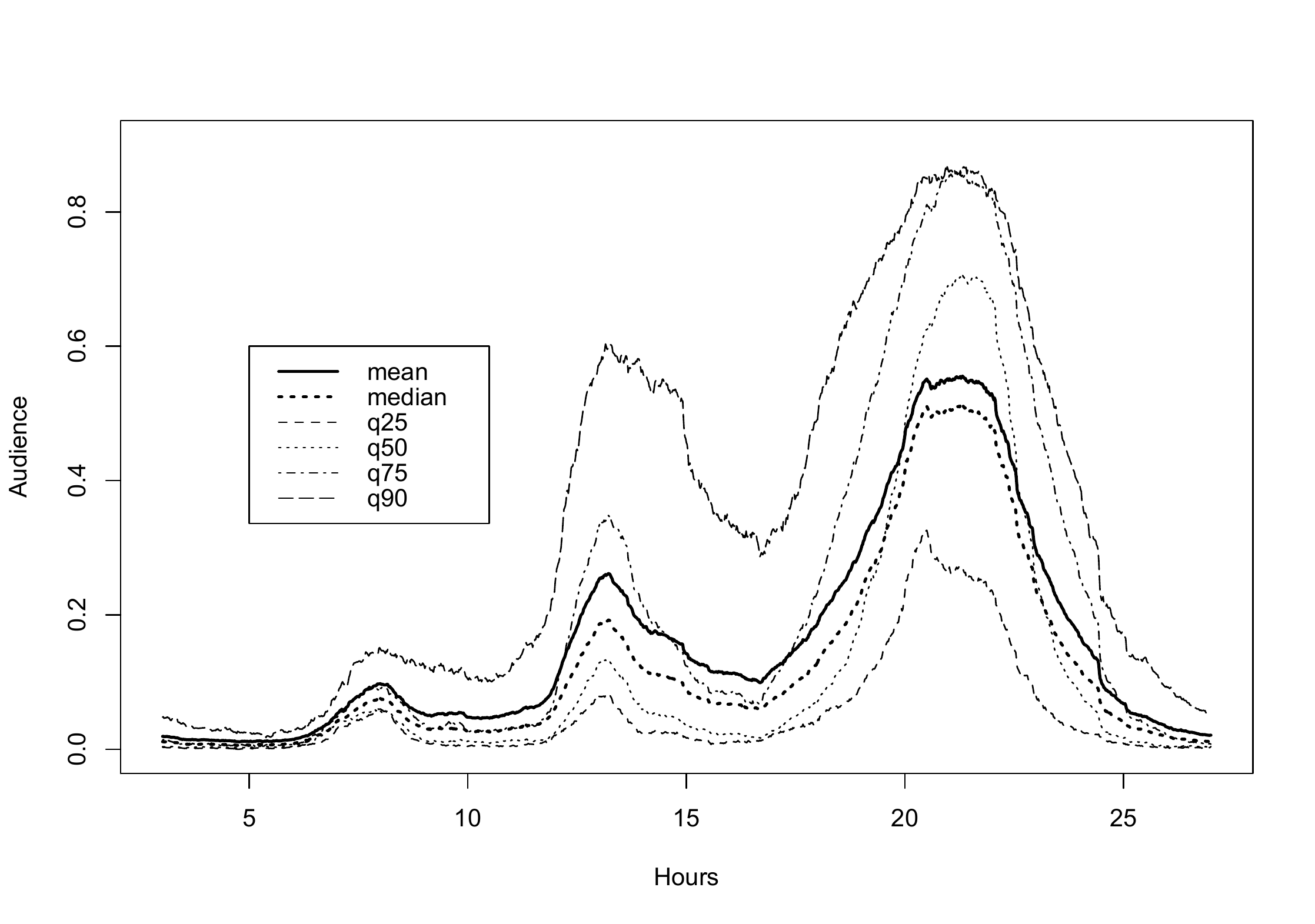}
 \caption{Estimation of the conditional median profile for different levels of total time spent watching television, on the 6th September 2010.}
 \label{fig1}
 \end{center}
   \end{figure}

\section{Proofs}
\begin{notation}
  In all the proofs, $x$ will be a fixed point in $\xR$ satisfying $p(x)>0$. 
  Since  $x$ will not vary, we will abuse notation and drop it from various quantities. In particular, in the following
  $m$ will denote the median $m(x)$ of the conditional law $\mu_x$, and we will write
  $Z_n = Z_n(x)$ and  $\Phi(\alpha) = \Phi(x,\alpha)$. 
\end{notation}
\subsection{About the assumptions}
\label{sec=wasserstein}
We begin by  a simple geometric result on unit vectors. For $a,b$
two points in $H$, let $D(a,b)$ be the unit vector ``starting'' from $a$ 
in the direction of $b$. 
 Now if $a,b,c$ are three points in $H$, such that $\nrm{a-b} \leq \nrm{a-c}$, Thales' theorem 
 shows that:
 \[ \frac{\nrm{D(a,b) - D(a,c)}}{\nrm{b - c'}} = \frac{\nrm{a + D(a,b) - a} }{\nrm{a-b}} = \frac{1}{\nrm{a-b}},\]
 so 
 \[ \nrm{D(a,b) - D(a,c)} \leq \frac{\nrm{b-c'}}{\nrm{a-b}} \leq \frac{\nrm{b-c}}{\nrm{a-b}}.\]

 \begin{center}
 \begin{tikzpicture}[point/.style={circle,fill,black,inner sep=0.1em}]
   \draw (5:8) -- (0,0)  -- (30:7);
   \node[point,label={below:$c$}] (c) at (5:7) {};
   \node[point,label={below:$a$}] (a) at (0,0) {};
   \node[point,label={above:$b$}] (b) at (30:6) {};
   \node[point,label={below:$a+D(a,c)$}] (ac) at (5:2) {};
   \node[point,label={above left:$a+D(a,b)$}] (ab) at (30:2) {};
   \node[point,label={below:$c'$}] (c') at (5:6) {};
   \draw[blue] (ac)--(ab);
   \draw[blue] (c')--(b);
 \end{tikzpicture}
 \end{center}
 
 In any case, 
 \begin{equation}
   \label{eq=Thales1}
   \nrm{D(a,b) - D(a,c)} \leq  \frac{\nrm{b-c}}{\min(\nrm{a-b},\nrm{a-c})}.
 \end{equation}
 We will need a ``decoupled'' version of this inequality:
 \begin{equation}
   \label{eq=Thales2}
   \nrm{D(a,b) - D(a,c)} 
   \leq  \frac{\nrm{b-c}}{\nrm{a-b}}
       + \frac{\nrm{b-c}}{\nrm{a-c}}.
 \end{equation}

  We can now prove that \assm{4} and \assm{6} imply \assm{3}. 
  Let $x,x'$ be two real numbers in the support of $p$. 
  Recall that $\mu_x$ denotes the law $\law(Y | X = x)$. 
  Let $Y$ and $Y'$ be two random variables with respective laws $\mu_x$ and $\mu_{x'}$, 
  such that their joint law $\pi$ achieves  the Wasserstein distance.  Let us first show that:
  \begin{equation}
    \label{eq=gradHolder}
  \forall \alpha \in H, \quad
  \nrm{\EE{D(\alpha,Y)} - \EE{D(\alpha,Y')}} \leq C \abs{x - x'}^\beta.
\end{equation}
  Fix an $\alpha\in H$. We have:
  \begin{align*}
    \nrm{\EE{D(\alpha,Y)} - \EE{D(\alpha,Y')}}
    &\leq \nrm{\E_\pi \left[ D(\alpha,Y) - D(\alpha,Y') \right] }  \\
    &\leq \E_\pi \left[ \nrm{D(\alpha,Y) - D(\alpha,Y')}\right].
  \end{align*}
  Now we use the geometric bound \eqref{eq=Thales2}, and Hölder's inequality:
  \begin{align*}
    \nrm{\EE{D(\alpha,Y)} - \EE{D(\alpha,Y')}}
    &\leq \E_\pi\left[ \frac{\nrm{Y - Y'}}{\nrm{Y - \alpha}}\right]
    + \E_\pi\left[ \frac{\nrm{Y - Y'}}{\nrm{Y'- \alpha}} \right] \\
    &\leq \left( \sqrt{\EE{\frac{1}{\nrm{Y-\alpha}^2}}}
               + \sqrt{\EE{\frac{1}{\nrm{Y' - \alpha}^2}}} \right)
	      \sqrt{ \E_\pi\left[ \nrm{Y-Y'}^2\right]}.
  \end{align*}
  The first term is bounded by $2\sqrt{\cstMoment}$ thanks to \assm{6}. 
  The second one is, by definition, the Wasserstein distance, and 
  is bounded by $\cstWass\abs{x - x'}^\beta$ thanks to \assm{4}, therefore 
  \eqref{eq=gradHolder} holds. Since $p$ is $\mathcal{C}^2$ with compact
  support, the product $\Phi(x,\alpha) = p(x) \EE{ D(\alpha, Y) \middle| X = x}$ is itself
  uniformly $\beta$-Hölder continuous; in other words \assm{3} holds. 

\subsection{First properties}
Recall that, for $z \in H,$ $\Phi_h(z)$ is defined by \eqref{def:PhiH} as the conditional expectation of the step, with window size $h$. 
When $h$ goes to zero, this ``expected step'' converges.
\begin{prop}
  \label{prp=PhiEtPhiH}The expected step is bounded:
  \begin{equation}
    \label{eq:PhiHBornee}
    \exists C, 
  \forall h>0, \forall \alpha \in H, \quad
  \nrm{\Phi_h(\alpha)} \leq p_{\max}.
\end{equation}
  Moreover,  under hypotheses \assm{2}, \assm{3} and  \assm{5},
  there exists a constant $C$ such that:
  \begin{equation}
    \nrm{\Phi_h(\alpha) - \Phi(\alpha)} \leq C h^\beta,
    \label{eq=Phi}
  \end{equation}
where $\Phi(x,\alpha)$ is defined by \eqref{eq=defPhi}.
\end{prop}
\begin{proof}
  With our strong hypotheses this result is easy to prove. Indeed 
  \[ \Phi_h(\alpha) = \int \frac{1}{h} K\left(\frac{x'-x}{h}\right) \Phi(x',\alpha)dx',\]
  so that by Jensen's inequality
  \[ \nrm{\Phi_h(\alpha)} \leq p_{\max} \int_{x'} \frac{1}{h} K\left( \frac{x - x'}{h} \right) dx' = p_{\max}.
  \]
  Moreover, 
  \[
    \nrm{\Phi_h(\alpha) - \Phi(\alpha)}
    \leq \int \frac{1}{h} K\left( \frac{x-x'}{h}\right) \nrm{\Phi(x',\alpha) - \Phi(x,\alpha)} dx'.
    \]
    Now we use Assumption \assm{3} to bound the norm by $\cstHolder\abs{x'-x}^\beta,$ the compact support of the bounded function $K$ (Assumption \textbf{A5}) and we integrate:
  \begin{align*}
    \nrm{\Phi_h(\alpha) - \Phi(\alpha)}
    &\leq \cstHolder\int \frac{1}{h} K\left( \frac{x-x'}{h}\right)\abs{x - x'}^\beta dx' \\
    &\leq \cstHolder \int  K(t) h^\beta t^\beta dt \\
    &\leq C h^\beta. \qedhere
  \end{align*}
\end{proof} 

Thanks to this result, we have a natural decomposition of algorithm \eqref{def:algo1}. Let us introduce the two following quantities:
\begin{align}
  \label{eq=defD}
  D_h(z) &= \Phi_h(z) - \Phi(z), \\
  \xi_{n+1} &= \left[ -\frac{Y_{n+1} - Z_{n}}{\nrm{Y_{n+1} - Z_{n}}} \frac{1}{h_n} K \left( \frac{X_{n+1}-x}{h_n}\right) - \Phi_{h_n}(Z_n) \right].
  \label{eq=defXi}
\end{align}
In terms of these quantities,  we can rewrite \eqref{def:algo1} as:
\begin{align}
  Z_{n+1} &= Z_{n} - \gamma_n \Phi(Z_n)  - \gamma_n D_{h_n}(Z_n) - \gamma_n \xi_{n+1}.
\label{def:algo2}
\end{align}
The first term $D_{h_n}(Z_n)$ will be controlled by Proposition~\ref{prp=PhiEtPhiH}. The second term $\xi_{n+1}$ defines  a sequence of martingale differences,
since the conditional expectation given  the sequence of $\sigma$-algebra $\CF_n=\sigma(Z_1,\ldots, Z_n)=\sigma(Y_1,X_1,\ldots, Y_n,X_n)$ satisfies
\begin{equation*}
  \EFn{ \xi_{n+1}} = 0, \ a.s.
\end{equation*}
For future reference, let us note the following bound on $\xi_n$:
\begin{align}
  \EFn{ \nrm{\xi_{n+1}}^2} &=  \EFn{ \nrm{ \frac{Y_{n+1} - Z_{n}}{\nrm{Y_{n+1} - Z_{n}}} \frac{1}{h_n} K \left( \frac{X_{n+1}-x}{h_n}\right) }^2} - \nrm{\Phi_{h_n}(Z_n)}^2, \  a.s \nonumber \\
   &\leq \frac{1}{h_n^2} \EE{K^2 \left( \frac{X_{n+1}-x}{h_n}\right)} - \nrm{\Phi_{h_n}(Z_n)}^2, \  a.s \nonumber \\
   \label{eq=crochetXi}
   &\leq \frac{C}{h_n}, \ a.s.
\end{align}

\subsection{Almost sure convergence}
In this section we prove Theorem~\ref{thm:cvps}. Define $V_n = \nrm{Z_n - m}^2$. By \eqref{def:algo2}, we have:
\begin{align*}
  V_{n+1}  & = V_n + \gamma_n^2\nrm{\Phi(Z_n)}^2 + \gamma_n^2 \nrm{D_{h_n}(Z_n) + \xi_{n+1}}^2 \\
  &\quad + 2 \scal{Z_n - m,  \gamma_n \Phi(Z_n)} + 2\gamma_n \scal{Z_n - m, D_{h_n}(Z_n) + \xi_{n+1}}  \\
  &\quad + 2 \gamma_n^2 \scal{\Phi(Z_n), D_{h_n}(Z_n) + \xi_{n+1} }.
\end{align*}
The first scalar product is non positive, we denote it by $(-\eta_n)$. We condition by $\Fn$: the $\xi_{n+1}$ in the scalar products disappear by the martingale property. 
Then we use Hölder's inequality:
\begin{align*}
  \EFn{ V_{n+1}}
  & \leq V_n + \gamma_n^2\nrm{\Phi(Z_n)}^2 + 2 \gamma_n^2 \nrm{D_{h_n}(Z_n)}^2 + 2\gamma_n^2 \EFn{ \nrm{\xi_{n+1}}^2} \\
  &\quad - \eta_n + 2\gamma_n \nrm{Z_n - m}\nrm{D_{h_n}(Z_n)}  + 2 \gamma_n^2 \nrm{\Phi(Z_n)}\nrm{ D_{h_n}(Z_n)}.
\end{align*}
On the last term we use $2 xy \leq x^2 + y^2$ to get:
\begin{align*}
  \EFn{ V_{n+1}}
  & \leq V_n + 2\gamma_n^2\nrm{\Phi(Z_n)}^2 + 3 \gamma_n^2 \nrm{D_{h_n}(Z_n)}^2 + 2\gamma_n^2 \EFn{ \nrm{\xi_{n+1}}^2} \\
  &\quad - \eta_n  + 2\gamma_n \nrm{Z_n - m}\nrm{D_{h_n}(Z_n)}.
\end{align*}
On the last term we bound $\nrm{Z_n - m}$ by $(1 + V_n)$ to get:
\begin{align*}
  \EFn{ V_{n+1}}
  & \leq\left(1 + 2 \gamma_n \nrm{D_{h_n}(Z_n)}\right) V_n 
  + 2\gamma_n^2\nrm{\Phi(Z_n)}^2 + 3\gamma_n^2 \nrm{D_{h_n}(Z_n)}^2  \\
  &\quad + 2\gamma_n^2 \EFn{ \nrm{\xi_{n+1}}^2}
  + 2\gamma_n \nrm{D_{h_n}(Z_n)} - \eta_n.
\end{align*}
Finally we bound $D_{h_n}(Z_n) $ by $Ch_n^\beta$ thanks to \eqref{eq=Phi}, $\EFn{\nrm{\xi_{n+1}}^2}$ by $C/h_n$ thanks to \eqref{eq=crochetXi}
and $\nrm{\Phi(Z_n)}$ by $p(x)$. This yields:
\begin{align*}
  \EFn{ V_{n+1}}
  & \leq\left(1 + 2C \gamma_n h_n^\beta\right) V_n 
  + 2p(x)\gamma_n^2 + 3 C^2\gamma_n^2 h_n^{2\beta}  \\
  &\quad + 2C\frac{\gamma_n^2 }{h_n}
  + 2C\gamma_nh_n^\beta  - \eta_n\\
  &\leq (1 + b_n)V_n + \chi_n - \eta_n,
\end{align*}
where $b_n = 2C \gamma_n h_n^\beta$ and $\chi_n = 3C^2\gamma_n^2h_n^{2\beta} + 2C\gamma_n^2h_n^{-1} +2 C\gamma_nh_n^\beta$ satisfy:
\begin{align*}
  \sum b_n &< \infty, & \sum \chi_n &<\infty.
\end{align*}
Therefore by the Robbins--Siegmund Lemma (Theorem 1.3.12 of \cite{Duf97}), $V_n$ converges almost surely and $\sum_n \eta_n < \infty$. This implies that the limit of $V_n$ is zero, by the same argument than in \cite{CCZ11}, assuming that $\sum_n \gamma_n=\infty$.

\subsection{Proof of proposition \ref{prp=vitesseQuadratique}}
For the sake of clarity, we follow the same steps as the proof of Proposition 3.2 in \cite{CCZ11}, 
and emphasize the necessary changes. 

\paragraph{Step 1 --- a spectral decomposition.}
This step is exactly the same as in \cite{CCZ11}: thanks to a spectral decomposition of $\Gamma$, 
we can define  the operators:
\begin{align*}
  \alpha_k &= \idt - \gamma_k \Gamma,
  &
  \beta_n  &= \alpha_n \alpha_{n-1} \cdots \alpha_1.
\end{align*}
Introducing the sequence  of real functions, for $n \in \xN,$ 
\[
f_n(x) = \prod_{k=1}^n (1- \gamma_k x),
\]
we see that  
each  operator $\beta_n$ can be also expressed  as follows:
\[
\beta_n x = \sum_{\lambda\in\Lambda} f_n(\lambda) \scal{e_\lambda,x} e_\lambda, \quad x \in H,
\]
their inverses are bounded operators, and satisfy: $\beta_n^{-1}x = \sum_{\lambda\in\Lambda} f_n^{-1}(\lambda) \scal{e_\lambda,x} e_\lambda.$ 

Moreover there exist constants $\kappa_1$, $\kappa_2, \kappa_3$ such that:
\begin{equation}
  \label{eq=asymptotiqueBetaN}
\begin{aligned}
 \forall x \in  \sigma(\Gamma), 
 \quad
 \kappa_1 \exp\left( - s_n x \right) &\leq f_n(x) \leq \kappa_2 \exp\left( -s_n x \right),  \\
 \abs{s_n  -  \frac{\cG}{1-\gamma} n^{1 - \gamma}} &\leq \kappa_3,
 \end{aligned}
 \end{equation}
 where we recall that $s_n = \sum_{k=1}^{n} \gamma_k$, and $\gamma_k = \cG k^{-\gamma}$. 

\paragraph{Step 2 --- Decomposition of the algorithm.}
Recall the decomposition \eqref{def:algo2} , and rewrite the algorithm as follows:
\begin{align}
  Z_{n+1}
  &= Z_n - \gamma_n \xi_{n+1} - \gamma_n\Phi(Z_n) - \gamma_n D_{h_n}(Z_n)
  \nonumber\\
  \label{def:algo3}
  &= Z_n - \gamma_n \xi_{n+1} - \gamma_n(\Gamma(Z_n-m) + \delta_n) - \gamma_n D_{h_n}(Z_n)
\end{align}
where $\delta_n = \Phi(Z_n) - \Gamma(Z_n-m)$ is the difference between the gradient
of $G$ and the gradient of its quadratic approximation. Compared to \cite{CCZ11}, 
there are two differences: the martingale difference $\xi_n$ has changed, 
and there is an additional term $\gamma_n D_{h_n}(Z_n)$. 
Therefore:
\begin{equation}
  \label{eq=decompositionI}
  \forall k,\quad
  Z_{k+1} - m = \alpha_k (Z_k - m)  - \gamma_k \xi_{k+1} - \gamma_k \delta_k - \gamma_k D_{h_k}(Z_k).
\end{equation}
Rewriting $\alpha_{n-1}\alpha_{n-2}\cdots \alpha_{k+1}$ as $\beta_{n-1}\beta_k^{-1}$, we get by induction, 
\begin{equation}
  \label{eq=decompositionII}
  Z_n-m=\beta_{n-1}(Z_1-m) + \beta_{n-1}M_n - \beta_{n-1}R_{n-1} - \beta_{n-1}R'_{n-1},
\end{equation}
where
\begin{align*}
  R_n &= \sum_{k=1}^{n-1}\gamma_k \beta_k^{-1}\delta_k  \\
M_n &= -\sum_{k=1}^{n-1}\gamma_k \beta_k^{-1}\xi_{k+1} \\
R'_n &= \sum_{k=1}^{n-1}\gamma_k \beta_k^{-1}D_{h_k}(Z_k).  \\
\end{align*}
At this point, the first and third term are the same as in \cite{CCZ11}, the martingale 
has changed  and there is an additional remainder term $R'_n$. 

\paragraph{Step 3 --- The deterministic term.} 
Just as in \cite{CCZ11}, we get:
\begin{equation}
  \label{eq=termeDeterministe}
  \EE{\nrm{\beta_{n-1}(Z_1 - m)}^2} 
  \leq C \exp\left( -2 n^{1-\gamma}\right) \EE{\nrm{Z_1 - m}^2}.
\end{equation}
\paragraph{Step 4 --- The martingale.}
\newcommand{\Ml}{M^\lambda}
Still following \cite{CCZ11}, we use the spectral decomposition to deal with the martingale part. The changes appear 
just before eq.\ (41) in that paper, where the bound on $\E[\nrm{\xi_k}^2]$ has to be changed (from $1$ to $C/h_n$, using
the new bound \eqref{eq=crochetXi}). Then we use the bounds \eqref{eq=asymptotiqueBetaN} to get:
\begin{align}
  \notag
  \E\left[\nrm{\beta_{n-1}M_n}^2\right]
  &\leq C \sum_{k\leq n-1} \frac{\gamma_k^2}{h_k} \left(
      \frac{f_{n-1}(\lmin)}{f_k(\lmin)}
    \right)^2 \\
    \label{eq=belleSomme} 
    &\leq C \sum_{k\leq n-1} \frac{\gamma_k^2}{h_k} \exp\left( - \frac{1}{1-\gamma} \left( n^{1 - \gamma} - k^{1- \gamma} \right) \right). 
\end{align}
Once more, the first terms in the sum are negligible (thanks to the exponential), and we
isolate the last terms, for $k\geq l(n)$,  where $l(n)$ is given by
\begin{equation}
  \label{eq=choixDeL}
l(n)^{1-\gamma} = n^{1 - \gamma} - c_\alpha\ln(n) \ , 
\end{equation}
for some constant $c_\alpha$. Choosing $c_\alpha$ large enough, the arguments from \cite{CCZ11} ensure
that the main contribution comes from the last terms. The number of terms, that is $n-l(n)$, is of the order $\ln(n)n^\gamma$, 
and $\gamma_{l(n)}^2 / h_{l(n)}$ is equivalent to $cn^{h - 2\gamma}$. 
Therefore
\begin{equation}
\label{eq=termeMarting}
\E\left[ \nrm{ \beta_{n-1} M_n}^2 \right] \leq  C \frac{\ln(n)}{n^{\gamma - h}}.
 \end{equation}

\paragraph{Step 5 --- the error terms.} \ \\

The first error term is  $R_n = \beta_{n-1} \sum_{k=1}^n \gamma_k \beta_k^{-1}\delta_k$, where $\delta_k = \Phi(Z_k) - \Gamma(Z_k -m)$. 
This one can be treated exactly as in \cite{CCZ11}. We recall 
the definition of the event $\Omega_N$: 
\[
  \Omega_N 
  = \left\{ \omega,
  \begin{array}{r}
    \forall n \geq N, \forall k\geq n - l(n), \quad  \nrm{Z_k(\omega) - m} \leq 1/K \\
    \qquad \text{ and } \nrm{\delta_k(\omega)} \leq C_r\nrm{Z_k(\omega) - m}^2\\
    \multicolumn{1}{l}{
    \forall k, \nrm{\delta_k(\omega)} \leq N.
    }
  \end{array}
  \right\},
\]
for a value of $K$ to be chosen later, and $l(n)$ defined by \eqref{eq=choixDeL}. 
Then, for any power of $n$ (say $n^{-42}$) there is a $C$ such that, on $\Omega_N$ and  for $n\geq N,$
\begin{equation}
  \label{eq=bidule}
  \nrm{\beta_{n-1}R_n}^2
  \leq 
  \frac{CN^2}{n^{42}} + \frac{C}{K^2} 
        \sum_{k=l(n) +1}^n \gamma_k \nrm{Z_k - m}^2.
\end{equation}

We now turn to the bound of the new error term $R'_n = \beta_{n-1} \sum_{k=1}^{n-1} \gamma_k \beta_k^{-1} D_{h_k}(Z_k)$. 
To bound $D_h$, we use \eqref{eq=Phi}:
\[ \nrm{D_h(z)} \leq C h^\beta. \]
Therefore for $N$ large enough, and for $k\geq l(n)$, 
\[ \nrm{\ind{\Omega_N} D_{h_k}(Z_k) }\leq C h_k^\beta.\]
For $k$ smaller than $l(n)$, we use the crude bound $D_{h_k}(Z_k) \leq p_{\max} + 1$. Finally we get:
\[
\nrm{\ind{\Omega_N} R'_n} \leq \frac{C}{n^{42}} + (n-l(n)) \gamma_{l(n)} h_{l(n)}^\beta. 
\]
The last term is bounded by $C n^{-\beta h}$ and dominates the first term. 

Finally, since by assumption, $h(1+2\beta) \geq \gamma$, one gets
\begin{equation}
  \nrm{\ind{\Omega_N} R'_n}^2 \leq \frac{C}{n^{\gamma - h}}.
  \label{eq=resteBis}
\end{equation}

Now we use \eqref{eq=termeDeterministe}, \eqref{eq=termeMarting}, \eqref{eq=bidule}
and \eqref{eq=resteBis} to bound the four terms that appear in  \eqref{eq=decompositionII}. We get, for $n\geq N$ and some new constant~$C$:
\begin{align*}
  \EE{\ind{\Omega_N} \nrm{Z_n - m}^2 }  
  &\leq \frac{C \ln(n)}{n^{\gamma-h}} + \frac{C'}{K^2} \sup_{l(n)<k\leq n} \EE{\ind{\Omega_N} \nrm{Z_k -m}^2}.
\end{align*}
By the same induction than in \cite{CCZ11}, 
we obtain the bound announced in Proposition~\ref{prp=vitesseQuadratique}. 

\subsection{Proof of Theorem \ref{thm=CLTavg}}
The following proof follows the same guidelines as the proof of Theorem 3.4 in \cite{CCZ11}. Again we emphasize the necessary changes due to the introduction of the kernel and of the conditional distribution. We first linearize the target function around the conditional median $m$ as in \eqref{def:algo3}:
\[
  \forall n,\quad
  Z_{n+1} - m = (\idt - \gamma_n \Gamma) (Z_n - m)  - \gamma_n \xi_{n+1} - \gamma_n \delta_n -\gamma_n D_{h_n}(Z_n),
\]
where $(\xi_n)$ is a martingale difference sequence. Therefore, for all $k$, 
\begin{equation}
  \label{eq=gragrigrou}
  \Gamma(Z_k - m) = \gamma_k^{-1}\left( (Z_k - m) - (Z_{k+1} - m)\right) - \xi_{k+1} - \delta_k -  D_{h_k}(Z_k).
\end{equation}
Define now,
\[
T_n \egaldef Z_n - m, \quad 
\Tbar_n\egaldef \Zbar_n - m 
\quad \text{and} \quad
M_{n+1}\egaldef \sum_{k=1}^n\xi_{k+1},
\]
and sum \eqref{eq=gragrigrou} over $k$
\[
n\Gamma\Tbar_n = 
\sum_{k=1}^n \frac{1}{\gamma_k}\left( T_k - T_{k+1} \right)
- \sum_{k=1}^n\left(\delta_k+D_{h_k}(Z_k)\right) - M_{n+1},
\]
so that
\begin{equation}
\label{dec-cle-clt}
\frac{n}{\sqrt{\sum_{k=1}^n\frac{1}{h_k}}}\Gamma\Tbar_n = 
\frac{n}{\sqrt{\sum_{k=1}^n\frac{1}{h_k}}}\left(
\frac{T_1\sqrt{h_1}}{\gamma_{1}}
- A_n + A'_n- A''_n
\right)
+ \frac{n}{\sqrt{\sum_{k=1}^n\frac{1}{h_k}}} M_{n+1},
\end{equation}
where
\begin{align*}
A_n &:= \frac{T_{n+1}}{\gamma_{n}},\\
A'_n &:=\sum_{k=2}^n T_k\left[\frac{1}{\gamma_k}-\frac{1}{\gamma_{k+1}}
  \right], \\
  A''_n &:=\sum_{k=1}^n\left(\delta_k+D_{h_k}(Z_k)\right).
\end{align*}

\paragraph{Step Zero --- convergence of covariance operators}
Our first task is to establish a central limit theorem for the last term of \eqref{dec-cle-clt}:
\begin{equation}
  \label{eq=martingaleCLT}
\frac{n}{\sqrt{\sum_{k=1}^n\frac{1}{h_k}}}M_n \cvl \mathcal{N}\left(0, \covLimite\right),
\end{equation}
where $\covLimite$ is the limiting covariance defined by \eqref{eq=defCovLimite}. 
On the space of linear operators  on $H$ we consider two classical norms, the (strong) operator norm
and the Hilbert-Schmidt norm:
\begin{align*}
  \nrmOp{A} = \sup\left\{ \nrm{Ay}_H  ; \nrm{y} \leq 1\right\}, \\
  \nrmHS{A} = \left( \sum_{i=0}^\infty \scal{Ae_j,e_j}^2 \right) ^{1/2},
  \label{eq=defNorms}
\end{align*}
where $e_j$ is an orthonormal base of $H$.  The following lemma will be useful.
\begin{lem}
  \label{lem:cvgDesCov}
  Define a random covariance operator $\Sigma_n$ by:
  \begin{equation}
    \Sigma_n = h_n \EFn{ \xi_{n+1} \otimes \xi_{n+1}}. 
    \label{eq=defSigmaN}
  \end{equation}
  Then:
  \begin{equation}
    \sqrt{\covLimite_n} \xrightarrow[n\to \infty]{\text{H.-S.}} \sqrt{\covLimite}, \quad a.s. 
    \label{eq=cvgSigmaN}
  \end{equation}
  In particular, $\covLimite_n$ converges to $\Sigma$ a.s.\ in the operator norm. 
  Moreover, if $\avgCov_n$ denotes the following averaged version of $\Sigma_n$:
  \begin{align*}
    \avgCov_n &= \frac{1}{\sum_{k=1}^n \frac{1}{h_k}} \sum_{k=1}^n \frac{1}{h_k}\Sigma_k  
    = \frac{1}{\sum_{k=1}^n \frac{1}{h_k}} \sum_{k=1}^n \ECond{ \xi_{k+1} \otimes \xi_{k+1}}{\CF_k}, 
  \end{align*}
  then
  \begin{equation}
    \sqrt{\avgCov_n} \xrightarrow[n\to \infty]{\text{H.-S.}} \sqrt{\covLimite}, \quad a.s. 
    \label{eq=cvgAvg}
  \end{equation}
  Finally, for any orthogonal projection operator $P$, 
  \begin{equation}
    \EE{\trace\left( \avgCov P\right)} \xrightarrow[n\to\infty]{} \EE{\trace\left(\covLimite P\right)}.
    \label{eq=cvgDesTraces}
  \end{equation}
\end{lem}
\begin{rem}
  Let us note that the convergence of square roots of covariance operators
  is equivalent to the convergence of the centered Gaussian laws with these
  covariances; see e.g. \citep{Bog98}, Example 3.8.13. 
\end{rem}
\begin{proof}
  We first show that the convergence \eqref{eq=cvgSigmaN} holds in operator norm. 
  Recall that $D(x,y)$ denotes the unit vector $(y-x)/\nrm{y-x}$. Let us rewrite
  $\covLimite_n$. 
\begin{equation*}
\covLimite_n = 
\frac{1}{h_n}\EFn{K^{2}\left(\frac{X_{n+1}-x}{h_n}\right) D(Z_n,Y_{n+1}) \otimes D(Z_n,Y_{n+1})} 
  - h_n\Phi_{h_n}(Z_n)\otimes\Phi_{h_n}(Z_n).
\end{equation*}
  Denote by $(X,Y)$ a couple of random variables with the original joint law, 
  and $Y_x$ be a random variable with law $\mu_x$, independent from $(X,Y)$. 
  
  We decompose the difference $\Sigma_n - \Sigma = D_1 + D_2 + D_3 + D_4$ where
  \begin{align*}
    D_1 &=  
\frac{1}{h_n}\EFn{K^{2}\left(\frac{X_{n+1}-x}{h_n}\right) D(Z_n,Y_{n+1}) \otimes D(Z_n,Y_{n+1})}
\\&\qquad
- \frac{1}{h_n}\EE{K^{2}\left(\frac{X-x}{h_n}\right) D(m,Y) \otimes D(m,Y)}  \\
    D_2 &= 
\frac{1}{h_n}\EE{K^{2}\left(\frac{X - x}{h_n}\right) D(m,Y) \otimes D(m,Y)}
- \frac{1}{h_n}\EE{K^{2}\left(\frac{X-x}{h_n}\right) D(m,Y_x) \otimes D(m,Y_x)}  \\
    D_3 &= 
 \frac{1}{h_n}\EE{K^{2}\left(\frac{X-x}{h_n}\right) D(m,Y_x) \otimes D(m,Y_x)} 
 - \covLimite \\
   D_4 &=  -h_n\Phi_{h_n}(Z_n)\otimes\Phi_{h_n}(Z_n).
  \end{align*}
  Note that only the 
  first and the last terms are random; the others depend on $n$ only through the quantity $h_n$. 
  For $(a,b,c)\in H^3$, it is easy to see that:
  \begin{align*}
    \nrmOp{D(a,b)\otimes D(a,b) - D(a,c) \otimes D(a,c)}
    &\leq  (\nrm{D(a,b)} + \nrm{D(a,c)}) \nrm{D(a,b) - D(a,c)} \\
    & \leq 2\left( \frac{1}{\nrm{a-b}} + \frac{1}{\nrm{a-c}}\right) \nrm{b-c},
  \end{align*}
  where we used \eqref{eq=Thales2} in the last line. 
  Therefore:
  \[
  \nrmOp{D_1} \leq 
  \frac{2}{h_n} \EE{K^2\left(\frac{X - x}{h_n} \right)
  \left( \frac{1}{\nrm{Y-m}} + \frac{1}{\nrm{Y - Z_n}}\right)} \nrm{Z_n - m}.
  \]
  Conditioning on $X$ and using Assumption \assm{6}, we get:
  \[
  \nrmOp{D_1} \leq 4\sqrt{\cstMoment} \frac{1}{h_n} \EE{ K^2( (X - x)/h_n)} \nrm{Z_n - m}. 
  \]
  The boundedness of $p$ and the finiteness of $v^2 = \int K^2(u)du$ ensure
  \begin{equation}
    \label{eq=cvgNoyau}
  \frac{1}{h_n} \EE{ K^2( (X - x)/h_n ) } \cv p(x) v^2, 
\end{equation}
  by dominated convergence; therefore the sequence $(h_n)^{-1} \EE{ K^2( (X - x)/h_n)}$ 
  is bounded. 
  Since $Z_n$ converges a.s.\ to $m$, $\nrmOp{D_1}$ converges a.s.\ to zero. 

  The second term $D_2$ is treated similarly; we get:
  \[
  \nrmOp{D_2} \leq 
  \frac{2}{h_n} \EE{K^2\left(\frac{X - x}{h_n} \right)
  \left( \frac{1}{\nrm{Y-m}} + \frac{1}{\nrm{Y_x - m }}\right) \nrm{Y - Y_x}}.
  \]
   Recall that $\mu_x = \law(Y|X = x)$ and
  let $\mu_{x,x'}$ be a coupling of $\mu_x$ and $\mu_{x'}$ that
  achieves the Wasserstein distance. We condition on the value of $X$ and we apply Hölder's inequality in order to 
  bound the first integral with Assumption \assm{6}
  and the second one with Assumption \assm{4}: 
  \begin{align*}
  \nrmOp{D_2} &\leq 
  \frac{2}{h_n} \EE{K^2\left(\frac{X - x}{h_n} \right)
  \left( \int \frac{1}{\nrm{y-m}^2} d\mu_X(y)\right)^{1/2}
  \left( \int \nrm{y-y'}^2 d\mu_{x,X} (y,y') \right)^{1/2}}
  \\&\quad
  + \frac{2}{h_n} \EE{K^2\left(\frac{X - x}{h_n} \right)
  \left( \int \frac{1}{\nrm{y-m}^2} \mu_x(dy)\right)^{1/2}
  \left( \int \nrm{y-y'}^2 d\mu_{x,X} (y,y') \right)^{1/2}} \\
  &\leq \frac{4\sqrt{\cstMoment}}{h_n} 
  \EE{K^2\left(\frac{X - x}{h_n} \right)
 \wass{ \mu_X, \mu_x}} \\
  &\leq \frac{4\cstWass\sqrt{\cstMoment}}{h_n} 
  \EE{K^2\left(\frac{X - x}{h_n} \right)
 \abs{X -x}^\beta} \\
 & = 4\cstWass \sqrt{\cstMoment}h_n^\beta \int K^2(y)\abs{y}^\beta dy =  \mathcal{O}(h_n^\beta) \cv 0. 
\end{align*}

  In the third term $D_3$, since $Y_x$ is independent of $X$ we may write
  \[
  D_3 = \left(\frac{1}{h_n} \EE{K^2( (X - x)/h_n)}  - p(x)v^2\right) \covLimite. 
  \]
  Thanks to \eqref{eq=cvgNoyau}, this converges to zero. Finally,
  by Proposition~\ref{prp=PhiEtPhiH}, $\Phi_{h_n}(Z_n)$ is almost surely bounded,   and since $\nrmOp{a\otimes b} \leq \nrm{a}\nrm{b}$, 
\[
\nrmOp{h_n\Phi_{h_n}(Z_n)\otimes\Phi_{h_n}(Z_n)}\leq h_n \nrm{\Phi_{h_n}(Z_n)}^2  \cv 0, \quad a.s.
\]
Therefore, $\covLimite_n$ converges to $\covLimite$ in the operator norm. 

To prove the convergence of $\avgCov_n$ in operator norm, observe that
\begin{align*}
  \nrmOp{\avgCov_n - \covLimite} &\leq \frac{1}{\sum \frac{1}{h_k}} \sum \frac{1}{h_k}\nrmOp{\covLimite_k - \covLimite}.
\end{align*}
Since $\nrmOp{\covLimite_k - \covLimite}$ converges to zero, the conclusion follows by the Toeplitz lemma. 

Let us show that these convergences hold in the Hilbert--Schmidt norm. 
For any $a$, $\trace(a\otimes a) = \nrm{a}$. Therefore:
  \begin{align*}
    \nrm{\sqrt{\covLimite_n}}_{H.S.}^2 &= \trace(\covLimite_n) \\
    &= \frac{1}{h_n} \EE{ K^2( (X - x)/h_n)} - h_n \nrm{\Phi_{h_n}(Z_n)}^2 \\
    & \cv  p(x) v^2  = \trace(\covLimite) = \nrmHS{\sqrt{\covLimite}}^2. 
  \end{align*}
Another application of the Toeplitz lemma shows that
  \begin{align*}  
  \nrm{\sqrt{\avgCov_n}}_{H.S.}^2
  = \trace(\avgCov_n) = \frac{1}{\sum_{k=1}^n \frac{1}{h_k}}\sum_{k=1}^n \trace\left(\frac{1}{h_k} \covLimite_k\right) 
  \xrightarrow[n\to\infty]{} \trace(\covLimite)  = \nrmHS{\sqrt{\covLimite}}^2. 
\end{align*}
  By the same reasoning as in Example 3.8.15 of \citep{Bog98}, 
  this implies the H.-S.\ convergences \eqref{eq=cvgSigmaN}
  and \eqref{eq=cvgAvg}. 

  Finally, let $P$ be an orthogonal projection operator. Choose a basis $(e_i)_{i\in \xN}$ of orthonormal eigenvectors of $P$: 
  $Pe_i = 0$ or $Pe_i = e_i$.  Since $\avgCov_n$ is trace-class, so is $\avgCov_n P$ and:
  \begin{align*}
    \trace\left( \avgCov_n P \right) &= \sum_i \scal{ e_i, \avgCov_n Pe_i} 
    = \sum_i \scal{Pe_i, \avgCov_n Pe_i} \\
    &= \sum_i \scal{\sqrt{\avgCov_n} P e_i, \sqrt{\avgCov_n} P e_i} 
    = \sum_i \nrm{\sqrt{\avgCov_n}P e_i}^2 \\
    &= \nrmHS{\sqrt{\avgCov_n} P}^2 \\
    &\xrightarrow[n\to \infty]{} \nrmHS{\covLimite P}^2 = \trace(\covLimite P). 
  \end{align*}
  This convergence is almost sure. Since $\trace(\covLimite_k) \leq \frac{1}{h_k} \EE{K^2( (X - x)/h_k)} \leq C$, 
  the convergence also holds in $L^1$ by dominated convergence, and \eqref{eq=cvgDesTraces} holds. 
\end{proof}

\paragraph{Step 1 --- The CLT for the martingale.} 
To prove the CLT \eqref{eq=martingaleCLT}, let us  check that the assumptions of Theorem 5.1 in \citep{Jak88} are fulfilled. Reminding (\ref{equiv}), translated in our context, these assumptions are:
\begin{equation}
\label{5-1-Jaku}
\forall \eta >0 \quad \lim_{n \to \infty} \PP\left(\sup_{1 \leq k \leq n} \sqrt{\frac{h_n}{n}}\nrm{\xi_{k+1}}>\eta \right)=0,
\end{equation}
\begin{equation}
\label{5-2-Jaku}
\mbox{a.s.} \quad \lim_{n \to \infty}\frac{1}{\sum_{k=1}^n\frac{1}{h_k}}\sum_{k=1}^n\scal{\xi_{k+1},e_i}\scal{\xi_{k+1},e_j} =\psi_{i,j}, 
\end{equation}
\begin{equation}
\label{5-3-Jaku}
\forall \varepsilon>0 \quad \lim_{N \to \infty}\limsup_{n \to \infty}\PP\left(\frac{h_n}{n}\sum_{i=1}^n\sum_{j=N}^{\infty}\scal{\xi_{i+1},e_j}^2>\varepsilon\right)=0,
\end{equation}
where $(e_n)_{n\in \xN}$ is an orthonormal basis of $H$ and $\psi_{i,j}:=\scal{\Sigma e_i,e_j}$.

We deal with condition  \eqref{5-1-Jaku} by applying Markov's inequality. Let $\eta>0$.
\begin{align*}
\PP\left(\sup_{1 \leq k \leq n} \sqrt{\frac{h_n}{n}}\nrm{\xi_{k+1}}>\eta \right)
&\leq  \sum_{k=1}^n\PP\left(\sqrt{\frac{h_n}{n}}\nrm{\xi_{k+1}}>\eta \right) \\
& \leq  \frac{1}{\eta^p} \sum_{k=1}^n\E\left[\left(\frac{h_n}{n}\right)^{p/2}\nrm{\xi_{k+1}}^p\right],
\end{align*}
for any $p\geq 1$. We chose an integer $p$ such that $p>2$. By convexity of the function $x \mapsto x^p$, we have, for any $n$,
\begin{equation*}
\nrm{\xi_{n+1}}^p \leq 2^{p-1}\left(\frac{1}{h_n^p}K^p\left(\frac{X_{n+1}-x}{h_n}\right)+\nrm{\Phi_{h_n}(Z_n)}^p\right).
\end{equation*}
Thus an easy computation yields
\begin{equation*}
\E\left[\nrm{\xi_{n+1}}^p\right] \leq \frac{2^{p-1}p_{\max}\int_{\xR}K^p(z)dz}{h_n^{p-1}}+2^{p-1}\E\left[\nrm{\Phi_{h_n}(Z_n)}^p\right].
\end{equation*}
In the last term, $\Phi_{h_n}(Z_n)$ is bounded, thanks to \eqref{eq:PhiHBornee}. 
Consequently, there exists a constant $C(p)$ (independent of $n$) such that 
\begin{equation*}
\E\left[\nrm{\xi_{n+1}}^p\right]\leq \frac{C(p)}{h_n^{p-1}}.
\end{equation*}
Hence we have, for a constant $C'(p)$ independent of $n$,
\begin{equation*}
\PP\left(\sup_{1 \leq k \leq n} \sqrt{\frac{h_n}{n}}\nrm{\xi_{k+1}}>\eta \right)  \leq  \frac{C(p)h_n^{p/2}}{n^{p/2}\eta^p} \sum_{k=1}^nh_k^{-p+1} \leq \frac{C'(p)}{n^{p/2-h(p/2-1)-1}}.
\end{equation*}
Since $p>2$, one has $p/2-h(p/2-1)-1>0$ and thus  \eqref{5-1-Jaku} holds.

\medskip 

Condition \eqref{5-2-Jaku} is a consequence of the law of large
numbers for martingales. 
Let us consider $(e_n)_{n \in \xN}$ an orthonormal basis of $H$. From the decomposition
\[
  \scal{\xi_{n+1},e_i}\scal{\xi_{n+1},e_j}
  = \EFn{\scal{\xi_{n+1},e_i}\scal{\xi_{n+1},e_j}}+\varepsilon_{n+1},
\]
with $\varepsilon_{n+1}:=\scal{\xi_{n+1},e_i}\scal{\xi_{n+1},e_j}-\EFn{\scal{\xi_{n+1},e_i}\scal{\xi_{n+1},e_j}}$, we have
\begin{align*}
\frac{1}{\sum_{k=1}^n\frac{1}{h_k}}\sum_{k=1}^n\scal{\xi_{k+1},e_i}\scal{\xi_{k+1},e_j}
&= \frac{1}{\sum_{k=1}^n\frac{1}{h_k}}\sum_{k=1}^n\E\left[\scal{\xi_{k+1},e_i}\scal{\xi_{k+1},e_j}|\CF_k\right]
  + \frac{1}{\sum_{k=1}^n\frac{1}{h_k}}\sum_{k=1}^n\ep_{k+1} \\
  &= \scal{e_i,\avgCov_n e_j} 
  + \frac{1}{\sum_{k=1}^n\frac{1}{h_k}}\sum_{k=1}^n\ep_{k+1}.
\end{align*}
By Lemma \ref{lem:cvgDesCov}, the matrix element $\scal{e_i,\avgCov_n e_j}$ 
converges to $\psi_{i,j}$. 
The law of large numbers for the martingale $\left(\sum_{k=1}^n\ep_{k+1}\right)_{n \in \xN}$ whose increasing process is of order $n^{1+3h}$ yields
\[
  \lim_{n \to \infty} \frac{1}{\sum_{k=1}^n\frac{1}{h_k}}\sum_{k=1}^n\ep_{k+1}=0 \quad \mbox{a.s.,} 
\]
since $h<1$, and condition \eqref{5-2-Jaku} is satisfied. 

\medskip

It remains to check condition \eqref{5-3-Jaku}.
Let $\varepsilon>0$. Applying Markov's inequality, we have
\begin{align*}
  \PP\left(\frac{h_n}{n}\sum_{k=1}^n\sum_{j=N}^{\infty}\scal{\xi_{k+1},e_j}^2>\varepsilon\right)
  &\leq \frac{h_n}{n\varepsilon}\sum_{k=1}^n\sum_{j=N}^{\infty}\E\left[\scal{\xi_{k+1},e_j}^2\right] \\
  &\leq \frac{h_n}{n\varepsilon} \frac{\sum \frac{1}{h_k}}{\sum \frac{1}{h_k}}
  \sum_{k=1}^n\sum_{j=N}^{\infty}\EE{\ECond{\scal{\xi_{k+1},e_j}^2}{\CF_k}} \\
  & \leq \frac{h_n}{n\varepsilon} \left(\sum_{k=1}^n \frac{1}{h_k}\right)
  \EE{\sum_{j=N}^\infty \scal{e_j, \avgCov_n e_j}}.
\end{align*}
Call $P_N$ the orthogonal projection on the $e_i$, $i\geq N$. 
\begin{align*}
  \PP\left(\frac{h_n}{n}\sum_{k=1}^n\sum_{j=N}^{\infty}\scal{\xi_{k+1},e_j}^2>\varepsilon\right)
  & \leq \frac{h_n}{n\varepsilon} \left(\sum_{k=1}^n \frac{1}{h_k}\right)
  \EE{\trace(\avgCov_n P_N)}.
\end{align*}
Therefore
\begin{align*}
 \limsup_n \PP\left(\frac{h_n}{n}\sum_{k=1}^n\sum_{j=N}^{\infty}\scal{\xi_{k+1},e_j}^2>\varepsilon\right)
 & \leq \frac{1}{(1+h) \varepsilon}
  \EE{\trace(\covLimite P_N)}, 
\end{align*}
and \eqref{5-3-Jaku} follows.

\paragraph{Step 2 --- The remaining terms are negligible.} 
Now, it remains to prove that all the other terms in (\ref{dec-cle-clt}) converge in probability to zero. Due to the equivalence (\ref{equiv}), we have to prove the convergence in probability to zero of
\[\sqrt{\frac{h_n}{n}}\left(
\frac{T_1\sqrt{h_1}}{\gamma_{1}}
- A_n + A'_n- A''_n
\right).\]
Recall that $\EE{\ind{\Omega_N} \nrm{T_n}^2} \leq C_N \frac{\ln(n)}{n^\gamma h_n},$ thanks to Proposition~\ref{prp=vitesseQuadratique}. 
For the first term $A_n = \frac{T_{n+1}}{\gamma_n}$, we have:
\begin{align*}
  \EE{\ind{\Omega_N}\frac{h_n}{n} \nrm{A_n}^2}
  &\leq \frac{C'_N \ln(n)}{n^{1-\gamma}},
\end{align*}
therefore $\sqrt{\frac{h_n}{n}}.A_n \cvp 0$. Let us turn to the second term $A'_n=\sum_{k=2}^n T_k\left[
    \frac{1}{\gamma_k}-\frac{1}{\gamma_{k+1}}
  \right]$. Since there exists a constant $C$ such that
\[ \abs{\frac{1}{\gamma_k} - \frac{1}{\gamma_{k+1}}}\leq C k^{\gamma - 1}, \]
by applying Jensen's inequality together with Proposition~\ref{prp=vitesseQuadratique}, there is  a positive constant $C$ such that
\begin{align*}
  \EE{\sqrt{\frac{h_n}{n}}\nrm{A'_n}\ind{\Omega_N}}
  &\leq C\sqrt{\ln(n)} n^{\gamma/2 - 1/2}.
\end{align*}
 Therefore $\sqrt{\frac{h_n}{n}}.A'_n\cvp 0$  since $\gamma < 1$. 

Finally, for the last term $A''_n=\sum_{k=1}^n\left(\delta_k+D_{h_k}(Z_k)\right)$, since on $\Omega_N$, $\nrm{\delta_k} \leq C_r \nrm{Z_k - m}^2$, we have for the part in $\delta_k$:
\begin{align*}
  \EE{\ind{\Omega_N} \nrm{\sqrt{\frac{h_n}{n}}\sum_{k=1}^n\delta_k}} 
  &\leq  C \ln(n)n^{h/2-\gamma+1/2},
\end{align*}
For the additional term, due to \eqref{eq=Phi}, we have $\EE{\ind{\Omega_N} \nrm{D_{h_k}(Z_k)}}\leq C h_k^{\beta}$ so that for some positive constant $C$,
\begin{align*}
  \EE{\ind{\Omega_N}\sqrt{\frac{h_n}{n}} \nrm{A''_n}} 
  &\leq C n^{1/2-h/2-h\beta}.
\end{align*}
The end of the proof follows the same guidelines as in \cite{CCZ11}. 

\bigskip

\noindent \textbf{Acknowledgements.} We thank the company M\'ediam\'etrie for allowing us to illustrate our methodologies with their data.

\bibliographystyle{apalike}
\bibliography{biblio_mediane}
\end{document}